\documentclass[12pt,oneside,reqno]{amsart}

\usepackage{amsmath,amssymb,amsthm}
\usepackage{geometry}
\usepackage{hyperref}
\hypersetup{hidelinks}
\newtheorem{theorem}{Theorem}[section]
\newtheorem{lemma}[theorem]{Lemma}
\newtheorem{proposition}[theorem]{Proposition}
\newtheorem{corollary}[theorem]{Corollary}
\theoremstyle{definition}
\newtheorem{definition}[theorem]{Definition}
\theoremstyle{remark}
\newtheorem*{remark*}{Remark}
\numberwithin{equation}{section}

\begin{document}

\title[Edge complexity of graphs]{Edge complexity of graphs}
\author{V. Gupta, A. Iosevich, J. Iosevich, B. Song, and H. Tian}
\thanks{The second listed author was supported in part by NSF grant DMS-2506858.}
\date{\today}

\begin{abstract}
Gupta and Iosevich introduced the edge complexity of a graph as the minimum Fourier ratio of its adjacency matrix over all vertex labelings and bounded it below by graph energy divided by the square root of twice the number of edges. We characterize equality for a fixed labeling: the Fourier transform of the adjacency matrix must have at most one nonzero entry in each row and column. This implies regularity, circulancy of every positive even power of an extremizing adjacency matrix, and a parity restriction on connected components, and it gives equality results for certain Laplacian spectral projectors. We construct equality cases from affine involutions on cyclic groups. Singer difference sets yield, for every prime power $q$, an equality-attaining $(q+1)$-regular graph that is not an abelian Cayley graph. We also establish Fourier-ratio estimates for weak, Cartesian, and strong graph products, including preservation of equality under weak products of coprime orders. We use Fourier-ratio recovery as a coding theorem to obtain entropy upper bounds for low-complexity adjacency matrices and complement them with a lower bound obtained by perturbing complete graphs. Finally, a concentration argument shows that if $Np_N/\log N\to\infty$ and $\limsup_{N\to\infty}p_N<1$, then $\operatorname{FR}_{\min}(G(N,p_N))$ is of order $N$ with probability tending to one.
\end{abstract}

\maketitle

\section{Introduction}

Let \(G=(V,E)\) be a simple graph with at least one edge and \(|V|=N\). After choosing a labeling of the vertices by \(\mathbb Z_N\), identify the adjacency matrix with the edge indicator
\[
f:\mathbb Z_N\times\mathbb Z_N\longrightarrow\{0,1\}.
\]
Its two-dimensional discrete Fourier transform is
\[
\widehat f(m,n)
=\frac1N\sum_{x,y\in\mathbb Z_N}f(x,y)e^{-2\pi i(mx+ny)/N}.
\]
Let \(F\) be the unitary Fourier matrix
\[
F_{mx}=\frac1{\sqrt N}e^{-2\pi i mx/N}.
\]
Then \(\widehat f=FfF\). Throughout the paper, \(\|M\|_1\) denotes the entrywise \(\ell^1\)-norm, \(\|M\|_2\) denotes the Frobenius norm, and \(\|M\|_{S_1}\) denotes the Schatten \(1\)-norm (nuclear norm). Parseval's identity gives \(\|\widehat f\|_2=\|f\|_2\).

The Fourier ratio of $f$ is
\[
\operatorname{FR}(f)=\frac{\|\widehat f\|_1}{\|\widehat f\|_2}.
\]
This quantity and its role in recovery, localization, and learning have been studied in \cite{Aldaleh2025,FRdiscrete}; see also \cite{Changpractical} for an additive-combinatorial application.
If \(f_\sigma\) is the adjacency matrix produced by a vertex labeling \(\sigma\), the edge complexity introduced in \cite{GuptaIosevich} is
\[
\operatorname{FR}_{\min}(G)
=\min_{\sigma\in S_N}\operatorname{FR}(f_\sigma).
\]
This invariant measures the best Fourier compressibility obtainable by relabeling the vertices.

Let \(L=D-A\) be the graph Laplacian. For a Laplacian eigenvalue \(\lambda\), let \(\Pi_{\lambda,\sigma}\) denote the matrix (equivalently, the kernel) of the orthogonal projector onto the \(\lambda\)-eigenspace under the labeling \(\sigma\). Its harmonic complexity is
\[
\operatorname{FR}_{\min}(\Pi_\lambda)
=\min_{\sigma\in S_N}
\frac{\|\widehat{\Pi_{\lambda,\sigma}}\|_1}
{\|\widehat{\Pi_{\lambda,\sigma}}\|_2}.
\]

The graph energy is
\[
\mathcal E(G)=\sum_{j=1}^N |\lambda_j(G)|=\|A\|_{S_1},
\]
where \(\lambda_1(G),\dots,\lambda_N(G)\) are the adjacency eigenvalues. Gupta and Iosevich proved the following lower bounds.

\begin{theorem}[\cite{GuptaIosevich}]\label{thm:energy-bound}
Let \(G\) be a graph of order \(N\) and size \(s>0\). Then
\[
\operatorname{FR}_{\min}(G)\geq \frac{\mathcal E(G)}{\sqrt{2s}}.
\]
Equality holds for every Cayley graph on \(\mathbb Z_N\) under its natural labeling.
\end{theorem}

\begin{theorem}[\cite{GuptaIosevich}]\label{thm:projector-bound}
Let \(\lambda\) be a Laplacian eigenvalue of multiplicity \(m(\lambda)\). Then
\[
\operatorname{FR}_{\min}(\Pi_\lambda)\geq \sqrt{m(\lambda)}.
\]
Equality holds for Cayley graphs on \(\mathbb Z_N\) under their natural labeling.
\end{theorem}

The first purpose of this paper is to understand equality in the energy bound. Theorem~\ref{thm:equality-criterion} gives a complete Fourier characterization for a fixed labeling. It converts equality between the entrywise \(\ell^1\)-norm and the nuclear norm into a weighted partial-permutation condition on the transformed adjacency matrix. This immediately imposes strong combinatorial and spectral restrictions. In particular, every equality-attaining graph is regular, and every positive even power of an extremizing adjacency matrix is circulant.

The second purpose is constructive. The affine-involution construction in Theorem~\ref{thm:affine-involution} includes cyclic Cayley graphs. Applied to Singer difference sets, it yields an infinite family of equality-attaining graphs with no four-cycles. Since every abelian Cayley graph of degree at least three contains a four-cycle, these examples cannot be abelian Cayley graphs.

The third purpose is to understand how edge complexity behaves under standard graph operations. For weak, Cartesian, and strong products we derive quantitative upper bounds from the product Fourier transform. These estimates are then compared with the corresponding graph energies. Weak products of equality-attaining graphs remain equality-attaining when the graph orders are coprime, while Cartesian and strong products remain within explicit factors of the energy lower bound. We also record a two-sided energy estimate for graph joins.

The fourth purpose is to quantify the size of the class of graphs with low edge complexity. We apply the recovery theorem below as a coding theorem: a common sparse sample recovers a positive proportion of all low-Fourier-ratio adjacency matrices, and the recovered matrices lie in small Hamming balls around a bounded collection of centers. This gives strong entropy upper bounds. We complement them with a lower bound obtained by perturbing the complete graph.

The fifth purpose is to determine the correct order of the minimum Fourier ratio of random graphs. If $Np_N/\log N\to\infty$ and $\limsup_{N\to\infty}p_N<1$, we prove that
\[
cN\leq\operatorname{FR}_{\min}(G(N,p_N))\leq N
\]
with probability tending to one; in particular, for every fixed $p\in(0,1)$, the minimum Fourier ratio of $G(N,p)$ is of order $N$, answering a question of \cite{GuptaIosevich}. The proof is based on concentration: the $\ell^1$-norm of the Fourier transform of the adjacency matrix is a convex Lipschitz function of the edge variables, and the resulting concentration inequality is strong enough for a union bound over all $N!$ vertex labelings.

We record the recovery result used in the counting argument.

\begin{theorem}[\cite{GuptaIosevich}]\label{thm:recovery}
Let \(f:\mathbb Z_N^2\to\mathbb C\) satisfy \(\|f\|_\infty\leq1\) and \(\operatorname{FR}(f)\leq r\). Fix \(\epsilon\in(0,1)\), and let \(X\subset\mathbb Z_N^2\) be obtained by retaining each point independently with probability \(p\). There exist absolute constants \(c,C>0\) such that, if
\[
pN^2\geq C\frac{r^2}{\epsilon^2}
\left(\log\frac r\epsilon\right)^2\log N,
\]
then, with probability at least \(1-e^{-cpN^2}\), the minimizer \(f^\ast\) selected by any fixed deterministic tie-breaking rule for
\[
\min_g\|\widehat g\|_1
\quad\text{subject to}\quad
\|g-f\|_{\ell^2(X)}\leq\epsilon\|f\|_2
\]
satisfies
\[
\|f^\ast-f\|_2\leq 11.47\epsilon\|f\|_2.
\]
\end{theorem}
\begin{remark*}
The cited theorem is stated in terms of ``the solution'' of the minimization problem. When the minimizer is not unique, we fix a deterministic tie-breaking rule once and for all. The proof of the recovery estimate uses only feasibility of \(f\) and the minimizing inequality
\[
\|\widehat{f^\ast}\|_1\leq\|\widehat f\|_1,
\]
so the same argument applies to the minimizer selected by any such fixed rule. This is the form needed in the counting argument of Section~\ref{section:enumeration}.
\end{remark*}

We will also need the notion of Hamming distance and Hamming balls in the counting argument.

\begin{definition}\label{def:hamming-distance}
For $A,B\in\{0,1\}^{N\times N}$, their Hamming distance is
\[
d_H(A,B)=\bigl|\{(i,j)\in\mathbb Z_N^2:A_{ij}\neq B_{ij}\}\bigr|.
\]
Thus entries are counted as ordered matrix positions.
\end{definition}

\begin{definition}\label{def:hamming-ball}
For $A\in\{0,1\}^{N\times N}$ and $R\geq0$, the Hamming ball of radius $R$ centered at $A$ is
\[
B_H(A,R)=\{B\in\{0,1\}^{N\times N}:d_H(A,B)\leq R\}.
\]
\end{definition}

\section{Equality cases of the energy and spectral bounds}\label{section3}
We now characterize and construct equality cases for Theorem~\ref{thm:energy-bound}. The central point is that the energy bound comes from the general inequality between the nuclear norm and the entrywise $\ell^1$-norm. Equality therefore has a rigid matrix-theoretic form.

\subsection{Conditions for equality}\label{section3.1}
We first give a Fourier criterion for equality in the bound.

The equality condition between the nuclear norm and the entrywise $\ell^1$-norm is proved by \cite[Lemma 3]{ZhuHayashiChen} in a different way. Here we include a short self-contained proof for completeness and to match the notation of the present paper.

\begin{theorem}\label{thm:equality-criterion}
Let \(G\) be a simple graph of order \(N\) and size \(s>0\), and let \(A_\sigma\) be its adjacency matrix under a fixed vertex labeling \(\sigma\). Then
\[
\operatorname{FR}(A_\sigma)=\frac{\mathcal E(G)}{\sqrt{2s}}
\]
if and only if every row and every column of \(\widehat{A_\sigma}=FA_\sigma F\) contains at most one nonzero entry.
\end{theorem}
\begin{proof}
Let \(M=\widehat{A_\sigma}\). Left and right multiplication by unitary matrices preserve singular values, so
\[
\|M\|_{S_1}=\|A_\sigma\|_{S_1}=\mathcal E(G).
\]
Parseval's identity and the fact that \(A_\sigma\) has exactly \(2s\) entries equal to one give
\[
\|M\|_2=\|A_\sigma\|_2=\sqrt{2s}.
\]
Consequently, equality in the energy bound for this labeling is equivalent to
\begin{equation}\label{eq:l1-nuclear-equality}
\|M\|_1=\|M\|_{S_1}.
\end{equation}

Choose a singular value decomposition \(M=U\Sigma V^\ast\), where \(U,V\) are unitary and \(\Sigma=\operatorname{diag}(\mu_1,\dots,\mu_N)\). Put \(W=UV^\ast\). Then \(W\) is unitary and
\[
\operatorname{tr}(W^\ast M)=\operatorname{tr}(V\Sigma V^\ast)=\sum_{j=1}^N\mu_j=\|M\|_{S_1}.
\]
Taking real parts and using \(|W_{mn}|\leq1\), we obtain
\[
\|M\|_{S_1}
=\operatorname{Re}\sum_{m,n}\overline{W_{mn}}M_{mn}
\leq\sum_{m,n}|W_{mn}||M_{mn}|
\leq\sum_{m,n}|M_{mn}|=\|M\|_1.
\]
If~\eqref{eq:l1-nuclear-equality} holds, then equality holds in the final inequality, and hence \(|W_{mn}|=1\) whenever \(M_{mn}\neq0\). A row of a unitary matrix cannot contain two entries of modulus one. Thus every row of \(M\) contains at most one nonzero entry. Applying the same argument to the columns gives the corresponding column condition.

Conversely, suppose every row and column of \(M\) contains at most one nonzero entry. The nonzero entries form a partial matching between rows and columns, which can be extended to a permutation \(\pi\in S_N\). If \(Q\) is the associated permutation matrix, then \(MQ\) is diagonal. Hence the singular values of \(M\) are the absolute values of its nonzero entries, and therefore
\[
\|M\|_{S_1}=\|M\|_1.
\]
This is equivalent to equality in the energy bound.
\end{proof}

For the remainder of this section, we use the Hermitian transform
\[
\widehat A:=FAF^\ast.
\]
Since $F^\ast=FF^2$ and $F^2$ is the permutation matrix implementing reflection on $\mathbb Z_N$, the matrices $FAF$ and $FAF^\ast$ differ only by a permutation of columns. Their entrywise $\ell^1$-norms, Frobenius norms, and row--column support conditions are therefore identical.

\begin{lemma}\label{lem:hermitian-transform}
If $A$ is an adjacency matrix, then $\widehat A=FAF^\ast$ is Hermitian.
\end{lemma}
\begin{proof}
Since $A^\ast=A$,
\[
\widehat A^\ast=(FAF^\ast)^\ast=FA^\ast F^\ast=FAF^\ast=\widehat A.
\]
\end{proof}

\begin{proposition}[Weighted involution form]\label{prop:weighted-involution}
Let $A$ be an extremizing adjacency matrix and put $H=FAF^\ast$. Then there are a subset $J\subset\mathbb Z_N$, an involution $\tau:J\to J$, and nonzero complex numbers $c_j$, $j\in J$, such that
\[
H_{j,k}=0\quad\text{unless }j\in J\text{ and }k=\tau(j),
\qquad
H_{j,\tau(j)}=c_j.
\]
Moreover,
\[
c_{\tau(j)}=\overline{c_j}.
\]
In particular, $c_j$ is real at every fixed point of $\tau$, while every two-cycle $\{j,\tau(j)\}$ contributes the pair of eigenvalues $|c_j|$ and $-|c_j|$.
\end{proposition}
\begin{proof}
By Theorem~\ref{thm:equality-criterion}, every row and column of $H$ contains at most one nonzero entry. Let $J$ be the set of indices of nonzero rows, and define $\tau(j)$ to be the column occupied by the unique nonzero entry in row $j$. Since $H$ is Hermitian, $H_{j,\tau(j)}\neq0$ implies $H_{\tau(j),j}\neq0$. The row--column uniqueness then gives $\tau(\tau(j))=j$, and Hermitian symmetry gives $c_{\tau(j)}=\overline{c_j}$. At a fixed point, the corresponding one-dimensional block is real. A two-cycle gives the block
\[
\begin{pmatrix}0&c_j\\ \overline{c_j}&0\end{pmatrix},
\]
whose eigenvalues are $\pm|c_j|$.
\end{proof}

\begin{corollary}\label{cor:rank-support}
Under an extremizing labeling, the number of nonzero entries of the Fourier transform of the adjacency matrix is exactly $\operatorname{rank}(A)$. Equivalently, equality in the energy bound forces the Fourier support to have the smallest size compatible with the rank.
\end{corollary}
\begin{proof}
In the weighted involution form, each nonzero row contains exactly one nonzero entry and the corresponding columns are distinct. Thus the nonzero columns are linearly independent, so the rank equals $|J|$, which is also the number of nonzero entries.
\end{proof}

\begin{lemma}[{\cite[Theorem 3.2.3]{Davis1979}}]\label{lem:diagonal-circulant}
    Let $D$ be a diagonal matrix. Then the inverse Fourier transform of $D$, or $F^\ast DF$, is circulant.
\end{lemma}

\begin{proposition}\label{prop:spectral-decomposition}
Let $G$ attain equality in Theorem~\ref{thm:energy-bound}, and let $A$ be an adjacency matrix under a labeling that attains equality. Then there are matrices $B$ and $C$ such that
\begin{itemize}
\item $A=B+C$;
\item $B$ is circulant;
\item the spectrum of $C$ is symmetric about $0$;
\item $C^2$ is circulant;
\item the column spaces of $B$ and $C$ are orthogonal.
\end{itemize}
\end{proposition}
\begin{proof}
Let $H=FAF^\ast$. By Lemma~\ref{lem:hermitian-transform}, $H$ is Hermitian, and by Theorem~\ref{thm:equality-criterion}, every row and column of $H$ contains at most one nonzero entry. Let $S$ be the diagonal part of $H$ and put $T=H-S$. Define
\[
B=F^\ast SF,
\qquad
C=F^\ast TF.
\]
Then $A=B+C$, and $B$ is circulant by Lemma~\ref{lem:diagonal-circulant}.

The matrix $T$ is Hermitian, has zero diagonal, and has at most one nonzero entry in each row and column. After a simultaneous permutation of rows and columns, it is therefore a direct sum of zero blocks and blocks of the form
\[
\begin{pmatrix}
0&z\\
\overline z&0
\end{pmatrix}.
\]
Its spectrum, and hence the spectrum of $C$, is symmetric about zero. The same block description shows that $T^2$ is diagonal. Thus
\[
C^2=F^\ast T^2F
\]
is circulant.

Finally, the support condition gives $ST=TS=0$: whenever $S_{ii}\neq0$, the diagonal entry is the unique nonzero entry in both row $i$ and column $i$ of $H$. Since $S$ and $T$ are Hermitian, their column spaces are orthogonal. Right multiplication by the invertible matrix $F$ does not change the column space, and left multiplication by $F^\ast$ applies the same unitary map to both column spaces. Hence the column spaces of $B$ and $C$ are orthogonal.
\end{proof}

We next formulate a combinatorial consequence of equality.

\begin{definition}
Let $G$ be a graph with a vertex labeling $\sigma:V(G)\to\mathbb Z_N$, and let $p_\ell(x,y)$ denote the number of walks of length $\ell$ from the vertex labeled $x$ to the vertex labeled $y$. For an integer $q\geq0$, we call $(G,\sigma)$ walk-circulant at length $q$ if
\[
p_q(x,y)=p_q(x+1,y+1)
\qquad\text{for all }x,y\in\mathbb Z_N.
\]
\end{definition}

\begin{theorem}\label{thm:square-circulant}
Let $G$ attain equality in Theorem~\ref{thm:energy-bound}, and let $A_\sigma$ be an adjacency matrix under a labeling that attains equality. Then $A_\sigma^2$ is circulant.
\end{theorem}
\begin{proof}
Put $H=FA_\sigma F^\ast$. The matrix $H$ is Hermitian, and every row and column contains at most one nonzero entry. If $i\neq j$, the supports of rows $i$ and $j$ are disjoint, because each column contains at most one nonzero entry. Therefore
\[
(H^2)_{ij}
=\sum_{k\in\mathbb Z_N}H_{ik}H_{kj}
=\sum_{k\in\mathbb Z_N}H_{ik}\overline{H_{jk}}
=0.
\]
Thus $H^2$ is diagonal. Since
\[
H^2=FA_\sigma^2F^\ast,
\]
we have $A_\sigma^2=F^\ast H^2F$, which is circulant by Lemma~\ref{lem:diagonal-circulant}.
\end{proof}

\begin{lemma}\label{lem:walk-circulant}
For a vertex labeling $\sigma$ of a graph $G$, the pair $(G,\sigma)$ is walk-circulant at length $q$ if and only if $A_\sigma^q$ is circulant.
\end{lemma}
\begin{proof}
The $(a,b)$-entry of $A_\sigma^q$ is the number of walks of length $q$ from the vertex labeled $a$ to the vertex labeled $b$. Thus the walk-circulant condition is exactly the defining translation identity for a circulant matrix.
\end{proof}

\begin{corollary}\label{cor:even-walk-circulant}
Let $G$ attain equality in Theorem~\ref{thm:energy-bound}, and let $\sigma$ be a labeling attaining equality. Then $(G,\sigma)$ is walk-circulant at length $2q$ for every positive integer $q$.
\end{corollary}
\begin{proof}
Write $A_\sigma^2=F^\ast DF$, where $D$ is diagonal. Then
\[
A_\sigma^{2q}=F^\ast D^qF,
\]
which is circulant by Lemma~\ref{lem:diagonal-circulant}. The result follows from Lemma~\ref{lem:walk-circulant}.
\end{proof}

\begin{corollary}\label{cor:regularity}
Every graph attaining equality in Theorem~\ref{thm:energy-bound} is regular.
\end{corollary}
\begin{proof}
Choose an extremizing adjacency matrix $A_\sigma$. By Theorem~\ref{thm:square-circulant}, $A_\sigma^2$ is circulant, so its diagonal entries are equal. These entries are precisely the vertex degrees.
\end{proof}

\begin{corollary}\label{cor:common-neighbors}
Let $A_\sigma$ be an extremizing adjacency matrix. There is a function $\gamma:\mathbb Z_N\to\mathbb Z_{\geq0}$ such that
\[
(A_\sigma^2)_{x,y}=\gamma(y-x)
\qquad\text{for all }x,y\in\mathbb Z_N.
\]
Thus, for distinct vertices, the number of common neighbors depends only on the difference of their labels. In particular, the degree is $\gamma(0)$.
\end{corollary}
\begin{proof}
This is the entrywise characterization of the circulant matrix $A_\sigma^2$ supplied by Theorem~\ref{thm:square-circulant}. For $x\neq y$, the entry $(A_\sigma^2)_{x,y}$ counts common neighbors of the vertices labeled $x$ and $y$.
\end{proof}

We next record a restriction on the connected components of an equality-attaining graph.

\begin{proposition}\label{prop:components}
Let $G$ attain equality in Theorem~\ref{thm:energy-bound}. Then there is a positive integer $v$ such that every connected component $C$ of $G$ is of one of the following two types:
\begin{itemize}
\item $C$ is nonbipartite and has $v$ vertices;
\item $C$ is bipartite, has $2v$ vertices, and its two bipartition classes both have $v$ vertices.
\end{itemize}
\end{proposition}
\begin{proof}
Choose a labeling $\sigma$ attaining equality and write $A=A_\sigma$. Since $G$ has at least one edge and is regular by Corollary~\ref{cor:regularity}, every vertex has positive degree.

Declare two vertices equivalent if there is an even-length walk between them. In a connected nonbipartite component there is a single equivalence class, whereas in a connected bipartite component the two equivalence classes are precisely the two bipartition classes.

For every equivalent pair $x,y$, choose an even walk from $x$ to $y$, of length $2r(x,y)$. Because every vertex has positive degree, a backtrack at the endpoint extends such a walk by two steps. There are only finitely many equivalent pairs, so for all sufficiently large $q$ there is a walk of length exactly $2q$ from $x$ to $y$ if and only if $x$ and $y$ are equivalent.

By Corollary~\ref{cor:even-walk-circulant}, $A^{2q}$ is circulant, so its zero--nonzero pattern is also circulant. Consider the graph on $\mathbb Z_N$ in which distinct vertices $x$ and $y$ are adjacent when $(A^{2q})_{xy}>0$. Its connected components are exactly the equivalence classes above. For completeness, if $K=\{t:(A^{2q})_{0t}>0\}$, the vertices connected to $0$ form the subgroup of $\mathbb Z_N$ generated by $K$; all other components are its cosets. Thus all equivalence classes have the same cardinality, say $v$, and the stated alternatives follow.
\end{proof}

We next give a consequence for Laplacian spectral projectors.

\begin{proposition}\label{prop:spectral-projector}
Let $G$ have a labeling $\sigma$ attaining equality in Theorem~\ref{thm:energy-bound}. Let $a$ be the common vertex degree, let $L=aI-A_\sigma$ be the Laplacian, and let $\mu$ be an eigenvalue of $A_\sigma$. Suppose that $|\mu|$ is different from the magnitude of every nonzero off-diagonal entry of $\widehat{A_\sigma}=FA_\sigma F^\ast$. Then $\lambda=a-\mu$ is an eigenvalue of $L$ and
\[
\operatorname{FR}_{\min}(\Pi_\lambda)=\sqrt{m(\lambda)}.
\]
\end{proposition}
\begin{proof}
Let $\Phi_{\mu,\sigma}$ be the orthogonal projector onto the $\mu$-eigenspace of $A_\sigma$, and let $m(\mu)$ denote the multiplicity of $\mu$. Since $L=aI-A_\sigma$, the $\lambda$-eigenspace of $L$, where $\lambda=a-\mu$, is exactly the $\mu$-eigenspace of $A_\sigma$. Thus
\[
\Pi_{\lambda,\sigma}=\Phi_{\mu,\sigma},
\qquad
m(\lambda)=m(\mu).
\]

The matrix $\widehat{A_\sigma}$ is Hermitian and has at most one nonzero entry in each row and column. After a simultaneous permutation of its rows and columns, it is therefore block diagonal with one-dimensional blocks and two-dimensional blocks of the form
\[
\begin{pmatrix}
0&b\\
\overline b&0
\end{pmatrix},
\qquad b\neq0.
\]
Each two-dimensional block contributes the eigenvalues $|b|$ and $-|b|$. By the hypothesis on $\mu$, the $\mu$-eigenspace is supported entirely on one-dimensional blocks. Consequently the spectral projector $\Psi_{\mu,\sigma}$ of $\widehat{A_\sigma}$ is diagonal, with exactly $m(\mu)$ diagonal entries equal to $1$ and all other entries equal to $0$.

Unitary conjugation gives
\[
\Psi_{\mu,\sigma}=F\Phi_{\mu,\sigma}F^\ast
=\widehat{\Phi_{\mu,\sigma}}.
\]
As noted above, replacing $F\Phi_{\mu,\sigma}F$ by $F\Phi_{\mu,\sigma}F^\ast$ only permutes the columns and therefore does not change the Fourier ratio. It follows that
\[
\operatorname{FR}(\Pi_{\lambda,\sigma})
=\operatorname{FR}(\Phi_{\mu,\sigma})
=\frac{\|\Psi_{\mu,\sigma}\|_1}{\|\Psi_{\mu,\sigma}\|_2}
=\frac{m(\mu)}{\sqrt{m(\mu)}}
=\sqrt{m(\lambda)}.
\]
The lower bound in Theorem~\ref{thm:projector-bound} proves that this labeling is minimizing.
\end{proof}

\begin{corollary}\label{cor:spectral-projector}
Let $G$ have a labeling $\sigma$ attaining equality in Theorem~\ref{thm:energy-bound}, let $a$ be its common vertex degree, and let $\Lambda$ be its Laplacian spectrum. If $\lambda\in\Lambda$ and $2a-\lambda\notin\Lambda$, then
\[
\operatorname{FR}(\Pi_{\lambda,\sigma})
=\operatorname{FR}_{\min}(\Pi_\lambda)
=\sqrt{m(\lambda)}.
\]
\end{corollary}
\begin{proof}
Set $\mu=a-\lambda$. Then $\mu$ is an adjacency eigenvalue and $-\mu$ is not. Every nonzero off-diagonal entry $b$ of $FA_\sigma F^\ast$ belongs to a two-dimensional Hermitian block that contributes both eigenvalues $|b|$ and $-|b|$. Consequently $|\mu|$ cannot equal the magnitude of a nonzero off-diagonal entry. Proposition~\ref{prop:spectral-projector} now applies.
\end{proof}

\subsection{Constructions attaining the energy bound}

We now turn to a sufficient condition. The construction yields bound-attaining graphs beyond the cyclic Cayley examples and, through the Singer family below, shows that equality is not confined to abelian Cayley graphs.

\begin{theorem}[Affine involution construction]\label{thm:affine-involution}
Let \(a\in\mathbb Z_N^\times\) satisfy
\[
 a^2\equiv1\pmod N.
\]
Let \(S\subset\mathbb Z_N\) be nonempty and assume
\[
 S=-aS,
 \qquad
 S\cap(1-a)\mathbb Z_N=\emptyset.
\]
Define
\[
 f(x,y)=
 \begin{cases}
 1,&\text{if }y-ax\in S,\\
 0,&\text{otherwise.}
 \end{cases}
\]
Then \(f\) is the edge indicator of a simple undirected graph \(G\) on \(\mathbb Z_N\). If \(s=|E(G)|\), then
\[
 \operatorname{FR}_{\min}(G)=\frac{\mathcal E(G)}{\sqrt{2s}}.
\]
\end{theorem}

\begin{proof}
Suppose \(f(x,y)=1\). Then there exists \(t\in S\) such that \(y-ax=t\). Combining with
\(a^2\equiv1\pmod N\), we have
\[
 x-ay=x-a(ax+t)=x-a^2x-at=-at.
\]
Since \(S=-aS\), we have \(-at\in S\). Hence \(f(y,x)=1\). Thus the corresponding adjacency matrix is symmetric.

If \(f(x,x)=1\), then \((1-a)x\in S\), which contradicts
\(S\cap(1-a)\mathbb Z_N=\emptyset\). Hence the corresponding graph is simple.

It remains to compute the Fourier transform. We have
\[
 \widehat f(m,n)=\frac1N\sum_{x,y\in\mathbb Z_N}\mathbf 1_S(y-ax)e^{-2\pi i(mx+ny)/N}.
\]
Set \(t=y-ax\), so \(y=ax+t\). Then
\[
 \begin{aligned}
 \widehat f(m,n)
 &=\frac1N\sum_{x\in\mathbb Z_N}\sum_{t\in S}e^{-2\pi i(mx+n(ax+t))/N}\\
 &=\frac1N\left(\sum_{x\in\mathbb Z_N}e^{-2\pi i(m+an)x/N}\right)
 \left(\sum_{t\in S}e^{-2\pi i nt/N}\right).
 \end{aligned}
\]
By character orthogonality,
\[
 \sum_{x\in\mathbb Z_N}e^{-2\pi i(m+an)x/N}
 =
 \begin{cases}
 N,&\text{if }m+an\equiv0\pmod N,\\
 0,&\text{otherwise.}
 \end{cases}
\]
Therefore
\[
 \widehat f(m,n)=\mathbf 1_{\{m+an\equiv0\pmod N\}}\sum_{t\in S}e^{-2\pi i nt/N}.
\]
Since \(a\in\mathbb Z_N^\times\), the congruence \(m+an\equiv0\pmod N\) has a unique solution
\(m\) for each \(n\), and a unique solution \(n\) for each \(m\). Hence every row and
every column of \(\widehat f\) contains at most one nonzero entry. The result follows
from Theorem~\ref{thm:equality-criterion}.
\end{proof}

Choosing specific involutions $a\in\mathbb Z_N^\times$ in
Theorem~\ref{thm:affine-involution} produces several families of examples. The case
$a=1$ recovers the Cayley graph equality case of \cite{GuptaIosevich}, while the case
$a=-1$ produces the sum-graph examples used below.

\begin{remark*}
The construction also allows nontrivial square roots of one when $N$ is composite. For instance, in $\mathbb Z_{15}$ one may take $a=4$ and $S=\{1,11\}$. Then $a^2\equiv1\pmod{15}$, the map $t\mapsto-4t$ interchanges $1$ and $11$, and
\[
(1-a)\mathbb Z_{15}=3\mathbb Z_{15}=\{0,3,6,9,12\}
\]
is disjoint from $S$. The resulting $2$-regular graph therefore attains equality. This illustrates that the affine-involution hypothesis is not restricted to the two choices $a=\pm1$.
\end{remark*}

\begin{corollary}[Cayley graphs]\label{cor:cyclic-cayley}
Let \(S\subset\mathbb Z_N\) be nonempty and assume
\[
 S=-S,
 \qquad
 0\notin S.
\]
Define
\[
 f(x,y)=\mathbf 1_S(y-x).
\]
Then \(f\) is the edge indicator of a simple undirected Cayley graph \(G\) on \(\mathbb Z_N\). If \(s=|E(G)|\), then
\[
 \operatorname{FR}_{\min}(G)=\frac{\mathcal E(G)}{\sqrt{2s}}.
\]
\end{corollary}

\begin{proof}
This is Theorem~\ref{thm:affine-involution} with \(a=1\). The hypotheses
\(S=-aS\) and \(S\cap(1-a)\mathbb Z_N=\emptyset\) become \(S=-S\) and \(0\notin S\), respectively.
The Fourier transform is supported on the line \(m+n=0\). This recovers the cyclic
Cayley equality case quoted in \cite{GuptaIosevich}.
\end{proof}

\begin{corollary}\label{cor:cyclic-sum}
Let \(S\subset\mathbb Z_N\) be nonempty and suppose
\[
 S\cap2\mathbb Z_N=\emptyset.
\]
Define
\[
 f(x,y)=\mathbf 1_S(x+y).
\]
Then \(f\) is the edge indicator of a simple undirected graph \(G\). If \(s=|E(G)|\), then
\[
 \operatorname{FR}_{\min}(G)=\frac{\mathcal E(G)}{\sqrt{2s}}.
\]
\end{corollary}

\begin{proof}
This is Theorem~\ref{thm:affine-involution} with \(a=-1\). The condition \(S=-aS\)
is automatic, and \(S\cap(1-a)\mathbb Z_N=\emptyset\) becomes \(S\cap2\mathbb Z_N=\emptyset\).
\end{proof}

\begin{corollary}\label{cor:Q3}
The three-dimensional hypercube \(Q_3\) is not a Cayley graph on $\mathbb Z_8$ and satisfies
\[
 \operatorname{FR}_{\min}(Q_3)=\frac{\mathcal E(Q_3)}{\sqrt{24}}=\sqrt6.
\]
\end{corollary}

\begin{proof}
Consider the graph on \(\mathbb Z_8\) defined by
\[
 f(x,y)=\mathbf 1_{\{1,3,5\}}(x+y).
\]

By Corollary~\ref{cor:cyclic-sum}, this graph attains the energy lower bound.

Define \(\sigma:\mathbb Z_8\to\{0,1\}^3\) by
\[
\begin{array}{c|cccccccc}
 x&0&1&2&3&4&5&6&7\\ \hline
 \sigma(x)&000&001&011&010&101&100&110&111.
\end{array}
\]

The unordered edges determined by \(x+y\in\{1,3,5\}\) are
\[
\begin{gathered}
\{0,1\},\{0,3\},\{0,5\},\{1,2\},\{1,4\},\{2,3\},\\
\{2,7\},\{3,6\},\{4,5\},\{4,7\},\{5,6\},\{6,7\}.
\end{gathered}
\]

Under \(\sigma\), each listed edge changes exactly one coordinate. Since there are \(12\) listed edges and \(Q_3\) has \(12\) edges, this graph is isomorphic to \(Q_3\).

Moreover, $Q_3$ is not a Cayley graph on $\mathbb Z_8$. Indeed, every inverse-closed three-element connection set in $\mathbb Z_8$ has the form $\{4,a,-a\}$. If $a$ is even, the resulting Cayley graph is disconnected. If $a$ is odd, it contains the $5$-cycle
\[
0,\ a,\ 2a,\ 3a,\ 4,\ 0.
\]
Neither possibility is consistent with the fact that $Q_3$ is connected and bipartite.

The adjacency spectrum of $Q_3$ is $\{3,1^{(3)},(-1)^{(3)},-3\}$, so $\mathcal E(Q_3)=12$ and $s=12$. Therefore
\[
 \operatorname{FR}_{\min}(Q_3)=\frac{12}{\sqrt{24}}=\sqrt6.
\]
\end{proof}

\subsection{The cube and its spectral projectors}\label{subsec:cube-projectors}

The three-dimensional cube is a small natural example showing that equality in the edge-energy bound does not force a cyclic Cayley representation. It also illustrates the scope and the limitations of Proposition~\ref{prop:spectral-projector}.

\begin{proposition}\label{prop:cube-projectors}
There is a labeling of the three-dimensional cube $Q_3$ by $\mathbb Z_8$ for which the adjacency Fourier transform has two one-dimensional blocks with eigenvalues $3$ and $-3$, while every nonzero off-diagonal entry has modulus $1$. Consequently the Laplacian projectors associated with the eigenvalues $0$ and $6$ satisfy
\[
\operatorname{FR}_{\min}(\Pi_0)=1,
\qquad
\operatorname{FR}_{\min}(\Pi_6)=1.
\]
The criterion in Proposition~\ref{prop:spectral-projector} does not determine whether the projectors associated with the eigenvalues $2$ and $4$ attain equality.
\end{proposition}
\begin{proof}
Use the sum-graph labeling constructed in Corollary~\ref{cor:Q3}, so that
\[
A(x,y)=\mathbf 1_{\{1,3,5\}}(x+y),
\qquad x,y\in\mathbb Z_8.
\]
For the Hermitian Fourier convention $H=FAF^\ast$, a direct change of variables gives
\[
H(m,n)=\mathbf 1_{\{m+n\equiv0\pmod 8\}}
\sum_{t\in\{1,3,5\}}e^{2\pi i nt/8}.
\]
The fixed points of the involution $n\mapsto-n$ are $0$ and $4$. At these indices the corresponding diagonal entries are $3$ and $-3$. At each of the remaining six indices, the displayed exponential sum has modulus $1$. Proposition~\ref{prop:spectral-projector} therefore applies to the adjacency eigenvalues $3$ and $-3$. Since $Q_3$ is $3$-regular, these correspond to the Laplacian eigenvalues $0$ and $6$. The adjacency eigenvalues $1$ and $-1$, corresponding to the Laplacian eigenvalues $2$ and $4$, have the same modulus as the off-diagonal entries, so the sufficient condition does not apply to them.
\end{proof}

\subsection{Levi graphs from Singer difference sets}\label{sec:singer}
In this subsection, we use difference sets to produce, for every prime power $q\geq2$, an equality-attaining graph that is not an abelian Cayley graph. The Heawood graph appears as the case $q=2$.

\begin{definition}[\cite{BJL}]\label{def:difference-set}
    Let $D\subset\mathbb Z_v$ with $|D|=k$. We call $D$ a $(v,k,\lambda)$-difference set in the cyclic group $\mathbb Z_v$ if every nonzero element of $\mathbb Z_v$ has exactly $\lambda$ representations of the form
    \[
        d-d',\qquad d,d'\in D.
    \]
    Here, ordered pairs are counted separately.
\end{definition}

\begin{theorem}[\cite{Singer1938}]\label{thm:singer-difference-set}
    Let $q$ be a prime power and put $v=q^2+q+1$. Then there exists a
    $(v,q+1,1)$-difference set $D\subset\mathbb Z_v$.
    Such a set is called a Singer difference set.
\end{theorem}

\begin{definition}\label{def:levi-graph}
    Let $D\subset\mathbb Z_v$. The Levi graph associated with $D$ is the graph with
    vertex set $P\cup L$, where
    \[
        P=\{P_i:i\in\mathbb Z_v\}
        \qquad\text{and}\qquad
        L=\{L_j:j\in\mathbb Z_v\}
    \]
    are two disjoint classes of $v$ vertices each, and in which $P_i$ is adjacent to $L_j$
    if and only if
    \[
        i+j\in D.
    \]
    Since every edge joins a vertex of $P$ to a vertex of $L$, the graph is bipartite with
    parts $P$ and $L$.
\end{definition}

To show that a given graph is not an abelian Cayley graph, we use the following lemma.

\begin{lemma}\label{lem:abelian-cayley-four-cycle}
Every simple undirected abelian Cayley graph of degree at least \(3\) contains a
\(4\)-cycle.
\end{lemma}
\begin{proof}
Let the graph be \(\operatorname{Cay}(G,S)\), where \(G\) is abelian, \(0\notin S\), and \(S=-S\). Its degree equals \(|S|\), so \(|S|\geq 3\).

Fix any \(c\in S\). Since \(|\{c,-c\}|\leq 2<3\leq|S|\), the set \(S\setminus\{c,-c\}\) is
nonempty. Thus there exists \(d\in S\) with \(d\neq c\) and \(d\neq -c\). We claim that the following chain
\[
 0\to c\to c-d\to -d\to0
\]
is a walk. The consecutive vertex differences are \(c,-d,-c,d\), all of which lie in \(S\) since
\(c,d\in S=-S\). Hence consecutive vertices are adjacent, and the last is adjacent to the first.

The four vertices are distinct. Indeed, $c,d\neq0$, $c\neq d$, and $c\neq-d$; these inequalities also imply $c-d\neq0$, $c-d\neq c$, and $c-d\neq-d$. Thus $0,c,c-d,-d$ form a $4$-cycle.
\end{proof}

\begin{theorem}\label{thm:singer-levi}
For every prime power \(q\geq2\), there is a graph \(G_q\) such that, with \(s_q=|E(G_q)|\),
\[
 \operatorname{FR}_{\min}(G_q)=\frac{\mathcal E(G_q)}{\sqrt{2s_q}},
\]
but \(G_q\) is not isomorphic to any abelian Cayley graph.
\end{theorem}

\begin{proof}
Let \(v=q^2+q+1\), and choose \(D\subset\mathbb Z_v\) as in
Theorem~\ref{thm:singer-difference-set}. Set
\[
 N=2v,
 \qquad
 S=1+2D\subset\mathbb Z_{2v}.
\]
Define a graph \(G_q\) on \(\mathbb Z_{2v}\) by
\[
 f(x,y)=\mathbf 1_S(x+y).
\]
Since every element of \(S\) is odd, Corollary~\ref{cor:cyclic-sum} gives
\[
 \operatorname{FR}_{\min}(G_q)=\frac{\mathcal E(G_q)}{\sqrt{2s_q}}.
\]
Now we write the even vertices as
\[
 P_i=2i,
 \qquad i\in\mathbb Z_v,
\]
and the odd vertices as
\[
 L_j=2j+1,
 \qquad j\in\mathbb Z_v.
\]
Then
\[
 P_i\sim L_j
 \quad\Longleftrightarrow\quad
 2i+(2j+1)\in1+2D
 \quad\Longleftrightarrow\quad
 i+j\in D.
\]
Thus \(G_q\) is the Levi graph associated with \(D\).
It is bipartite and \((q+1)\)-regular.

We next prove that \(G_q\) has no \(4\)-cycle. Suppose two distinct vertices
\(P_i,P_{i'}\) had two distinct common neighbors \(L_j,L_{j'}\). Then
\[
 i+j,
 \quad i+j',
 \quad i'+j,
 \quad i'+j'
\]
all belong to \(D\). Moreover
\[
 (i+j)-(i+j')=(i'+j)-(i'+j').
\]
Since \(j\neq j'\), this is a nonzero element of \(\mathbb Z_v\). By the uniqueness
property in Theorem~\ref{thm:singer-difference-set}, the two ordered representations
must be the same. Hence
\[
 i+j=i'+j,
 \qquad
 i+j'=i'+j',
\]
and therefore \(i=i'\), a contradiction. Thus no \(4\)-cycle exists.

Since \(q\geq2\), the degree \(q+1\) is at least \(3\). By
Lemma~\ref{lem:abelian-cayley-four-cycle}, every simple undirected abelian Cayley graph
of degree at least \(3\) contains a \(4\)-cycle. Therefore \(G_q\) is not an abelian
Cayley graph.
\end{proof}

\begin{proposition}\label{prop:energy-singer}
The adjacency spectrum of $G_q$ is
\[
\{(q+1)^{(1)},\ (\sqrt q)^{(v-1)},\ (-\sqrt q)^{(v-1)},\ (-(q+1))^{(1)}\},
\]
where $v=q^2+q+1$. In particular,
\[
\mathcal E(G_q)=2(q+1)+2(v-1)\sqrt q.
\]
\end{proposition}
\begin{proof}
Let $M$ be the $v\times v$ bipartite incidence matrix defined by
\[
M_{i,j}=1 \quad\Longleftrightarrow\quad i+j\in D.
\]
For $i,i'\in\mathbb Z_v$, the $(i,i')$-entry of $MM^T$ counts the number of $j$ such that both $i+j$ and $i'+j$ belong to $D$. This number is $q+1$ when $i=i'$ and is $1$ when $i\neq i'$, by the difference-set property. Hence
\[
MM^T=qI+J.
\]
The singular values of $M$ are therefore $q+1$ with multiplicity $1$ and $\sqrt q$ with multiplicity $v-1$. Since the adjacency matrix of $G_q$ is
\[
\begin{pmatrix}
0&M\\
M^T&0
\end{pmatrix},
\]
its eigenvalues are the positive and negative singular values of $M$, with the asserted multiplicities. Summing their absolute values gives the energy formula.
\end{proof}

The smallest member of this family is the Heawood graph.

\begin{corollary}\label{cor:q2-heawood}
The graph $G_2$ in Theorem~\ref{thm:singer-levi} is the Heawood graph. Consequently,
\[
\operatorname{FR}_{\min}(G_2)
=\frac{6+12\sqrt2}{\sqrt{42}}.
\]
\end{corollary}
\begin{proof}
For $q=2$, we have $v=7$, and
\[
D=\{0,2,6\}\subset\mathbb Z_7
\]
is a $(7,3,1)$-difference set: its six nonzero ordered differences are precisely the six nonzero elements of $\mathbb Z_7$. The associated symmetric $(7,3,1)$ design is the projective plane of order $2$, namely the Fano plane, and its Levi graph is the Heawood graph; see, for example, \cite{BJL}. Proposition~\ref{prop:energy-singer} gives
\[
\mathcal E(G_2)=6+12\sqrt2.
\]
Since $G_2$ is $3$-regular on $14$ vertices, it has $21$ edges, and Theorem~\ref{thm:singer-levi} completes the calculation.
\end{proof}

\section{Edge complexity and graph products}\label{section:products}

Graph products provide a systematic way of generating larger graphs from smaller ones. In the present setting a minor arithmetic issue arises: edge complexity is defined using the cyclic group $\mathbb Z_N$, whereas the natural labeling of a product graph is by $\mathbb Z_{N_1}\times\mathbb Z_{N_2}$. We therefore assume throughout the Fourier-ratio statements in this section that $\gcd(N_1,N_2)=1$. The Chinese remainder theorem then identifies these groups, and the Fourier transform on $\mathbb Z_{N_1N_2}$ agrees with the product Fourier transform up to independent permutations of the physical and frequency variables. Such permutations preserve the entrywise $\ell^1$-norm and the Frobenius norm. Throughout this section, $\widehat A$ again denotes the original transform $FAF$.

\subsection{Fourier-ratio bounds}

\begin{definition}
Let $G$ and $H$ be graphs. Their weak product $G\times H$ has vertex set $V(G)\times V(H)$, and $(g_1,h_1)$ is adjacent to $(g_2,h_2)$ if and only if $g_1$ is adjacent to $g_2$ in $G$ and $h_1$ is adjacent to $h_2$ in $H$.
\end{definition}

\begin{definition}
The Cartesian product $G\square H$ has vertex set $V(G)\times V(H)$, and $(g_1,h_1)$ is adjacent to $(g_2,h_2)$ if and only if either $g_1=g_2$ and $h_1$ is adjacent to $h_2$, or $h_1=h_2$ and $g_1$ is adjacent to $g_2$.
\end{definition}

\begin{definition}
The strong product $G\boxtimes H$ is obtained by taking the union of the weak and Cartesian products on the common vertex set $V(G)\times V(H)$.
\end{definition}

\begin{proposition}\label{prop:weak-product-fr}
Let $G$ and $H$ be graphs of coprime orders $N_1$ and $N_2$, each with at least one edge. Then
\[
\operatorname{FR}_{\min}(G\times H)
\leq
\operatorname{FR}_{\min}(G)\operatorname{FR}_{\min}(H).
\]
\end{proposition}
\begin{proof}
Choose minimizing labelings of $G$ and $H$ by $\mathbb Z_{N_1}$ and $\mathbb Z_{N_2}$. Under the induced product labeling, the edge indicator of $G\times H$ is the tensor product of the two edge indicators. The Chinese remainder theorem transports this labeling to $\mathbb Z_{N_1N_2}$. On the Fourier side, the transformed adjacency matrix is, up to independent permutations of rows and columns,
\[
\widehat A_G\otimes\widehat A_H.
\]
Both the entrywise $\ell^1$-norm and the Frobenius norm multiply under tensor products. Hence the Fourier ratio of this labeling is
\[
\operatorname{FR}(A_G)\operatorname{FR}(A_H),
\]
and minimizing over all labelings of the product proves the result.
\end{proof}

\begin{proposition}\label{prop:cartesian-product-fr}
Let $G$ and $H$ be graphs of coprime orders $N_1,N_2$ and positive edge counts $s_G,s_H$. Then
\[
\operatorname{FR}_{\min}(G\square H)
\leq
\frac{N_2\sqrt{s_G}\operatorname{FR}_{\min}(G)
+N_1\sqrt{s_H}\operatorname{FR}_{\min}(H)}
{\sqrt{N_2s_G+N_1s_H}}.
\]
\end{proposition}
\begin{proof}
Under the product labeling,
\[
A_{G\square H}=A_G\otimes I_{N_2}+I_{N_1}\otimes A_H.
\]
Under this convention, the Fourier transform of $I_N$ is the reflection permutation matrix $F^2$. In particular,
\[
\|\widehat I_N\|_1=N,
\qquad
\|\widehat I_N\|_2=\sqrt N.
\]
After applying the Chinese remainder identification and the triangle inequality,
\[
\|\widehat A_{G\square H}\|_1
\leq
N_2\|\widehat A_G\|_1+N_1\|\widehat A_H\|_1.
\]
The Cartesian product has $N_2s_G+N_1s_H$ edges, so Parseval gives
\[
\|\widehat A_{G\square H}\|_2
=\sqrt{2(N_2s_G+N_1s_H)}.
\]
Substitution of
\(
\|\widehat A_G\|_1=\sqrt{2s_G}\operatorname{FR}_{\min}(G)
\)
and the analogous identity for $H$ proves the claim.
\end{proof}

\begin{proposition}\label{prop:strong-product-fr}
Let $G$ and $H$ be graphs of coprime orders $N_1,N_2$ and positive edge counts $s_G,s_H$. Then
\[
\operatorname{FR}_{\min}(G\boxtimes H)
\leq
\frac{
N_2\sqrt{s_G}\operatorname{FR}_{\min}(G)
+N_1\sqrt{s_H}\operatorname{FR}_{\min}(H)
+\sqrt{2s_Gs_H}\operatorname{FR}_{\min}(G)\operatorname{FR}_{\min}(H)}
{\sqrt{N_2s_G+N_1s_H+2s_Gs_H}}.
\]
\end{proposition}
\begin{proof}
The adjacency matrix is
\[
A_{G\boxtimes H}
=A_G\otimes I_{N_2}+I_{N_1}\otimes A_H+A_G\otimes A_H.
\]
The same product-Fourier calculation used above gives
\[
\|\widehat A_{G\boxtimes H}\|_1
\leq
N_2\|\widehat A_G\|_1
+N_1\|\widehat A_H\|_1
+\|\widehat A_G\|_1\|\widehat A_H\|_1.
\]
The three edge sets in the definition of the strong product are disjoint, and therefore
\[
s_{G\boxtimes H}=N_2s_G+N_1s_H+2s_Gs_H.
\]
Dividing by the Frobenius norm $\sqrt{2s_{G\boxtimes H}}$ gives the stated estimate.
\end{proof}

\subsection{Comparison with the energy bound}

For a graph $G$ with edge count $s_G>0$, define its energy-gap factor
\[
\kappa(G)
=
\frac{\operatorname{FR}_{\min}(G)}
{\mathcal E(G)/\sqrt{2s_G}}.
\]
Theorem~\ref{thm:energy-bound} says that $\kappa(G)\geq1$, and equality is precisely the equality problem studied in Section~\ref{section3}.

\begin{proposition}\label{prop:weak-product-energy}
Let $G$ and $H$ be graphs of coprime orders, each with at least one edge. Then
\[
\kappa(G\times H)\leq\kappa(G)\kappa(H).
\]
In particular, if both factors attain equality in Theorem~\ref{thm:energy-bound}, then so does $G\times H$.
\end{proposition}
\begin{proof}
The adjacency matrix of the weak product is $A_G\otimes A_H$. Hence
\[
\mathcal E(G\times H)=\mathcal E(G)\mathcal E(H),
\qquad
s_{G\times H}=2s_Gs_H.
\]
Combining these identities with Proposition~\ref{prop:weak-product-fr} gives the first assertion. If $\kappa(G)=\kappa(H)=1$, the resulting upper bound for $\operatorname{FR}_{\min}(G\times H)$ coincides with the universal energy lower bound, proving equality.
\end{proof}

\begin{proposition}\label{prop:cartesian-energy}
Let $G$ and $H$ be graphs of orders $N_1$ and $N_2$, respectively, each with at least one edge. Then
\[
\mathcal E(G\square H)
\geq
\frac12\bigl(N_2\mathcal E(G)+N_1\mathcal E(H)\bigr).
\]
If in addition $N_1$ and $N_2$ are coprime, then
\[
\kappa(G\square H)
\leq2\max\{\kappa(G),\kappa(H)\}.
\]
\end{proposition}
\begin{proof}
Let $\lambda_1,\dots,\lambda_{N_1}$ and $\mu_1,\dots,\mu_{N_2}$ be the adjacency eigenvalues of $G$ and $H$. The eigenvalues of the Cartesian product are $\lambda_i+\mu_j$. Since the adjacency matrix of $H$ has trace zero, for each fixed $i$ the triangle inequality gives
\[
\sum_{j=1}^{N_2}|\lambda_i+\mu_j|
\geq
\left|\sum_{j=1}^{N_2}(\lambda_i+\mu_j)\right|
=N_2|\lambda_i|.
\]
Summing in $i$ yields
\[
\mathcal E(G\square H)\geq N_2\mathcal E(G).
\]
By symmetry,
\[
\mathcal E(G\square H)\geq N_1\mathcal E(H).
\]
Consequently,
\[
\mathcal E(G\square H)
\geq
\max\{N_2\mathcal E(G),N_1\mathcal E(H)\}
\geq
\frac12\bigl(N_2\mathcal E(G)+N_1\mathcal E(H)\bigr).
\]
When the orders are coprime, Proposition~\ref{prop:cartesian-product-fr} gives
\[
\operatorname{FR}_{\min}(G\square H)
\leq
\frac{N_2\kappa(G)\mathcal E(G)+N_1\kappa(H)\mathcal E(H)}
{\sqrt{2s_{G\square H}}}.
\]
The numerator is at most
\[
2\max\{\kappa(G),\kappa(H)\}\mathcal E(G\square H),
\]
which proves the second assertion.
\end{proof}

\begin{proposition}\label{prop:strong-energy}
Let $G$ and $H$ be graphs of orders $N_1$ and $N_2$, respectively, each with at least one edge, and define
\[
r(G,H)=\frac{N_1}{\mathcal E(G)}+\frac{N_2}{\mathcal E(H)}.
\]
Then
\[
\mathcal E(G\boxtimes H)
\geq
\mathcal E(G)\mathcal E(H)\bigl(1-r(G,H)\bigr).
\]
If the orders are coprime and $r(G,H)<1$, then
\[
\kappa(G\boxtimes H)
\leq
\kappa(G)\kappa(H)
\frac{1+r(G,H)}{1-r(G,H)}.
\]
\end{proposition}
\begin{proof}
The eigenvalues of the strong product are
\[
\lambda_i\mu_j+\lambda_i+\mu_j.
\]
The reverse triangle inequality therefore gives
\begin{align*}
\mathcal E(G\boxtimes H)
&\geq
\sum_{i,j}\bigl(|\lambda_i\mu_j|-|\lambda_i|-|\mu_j|\bigr)\\
&=
\mathcal E(G)\mathcal E(H)-N_2\mathcal E(G)-N_1\mathcal E(H)\\
&=
\mathcal E(G)\mathcal E(H)\bigl(1-r(G,H)\bigr).
\end{align*}
Assume now that the orders are coprime. Since $\kappa(G),\kappa(H)\geq1$, Proposition~\ref{prop:strong-product-fr} implies
\[
\operatorname{FR}_{\min}(G\boxtimes H)
\leq
\frac{\kappa(G)\kappa(H)\mathcal E(G)\mathcal E(H)(1+r(G,H))}
{\sqrt{2s_{G\boxtimes H}}}.
\]
When $r(G,H)<1$, the preceding energy estimate converts this into the stated bound for $\kappa(G\boxtimes H)$.
\end{proof}

\begin{corollary}\label{cor:iterated-weak-products}
Let $G_1,\dots,G_k$ be equality-attaining graphs whose orders are pairwise coprime. Then the iterated weak product
\[
G_1\times\cdots\times G_k
\]
also attains equality in Theorem~\ref{thm:energy-bound}.
\end{corollary}
\begin{proof}
Apply Proposition~\ref{prop:weak-product-energy} inductively. Pairwise coprimality ensures that at every stage the order of the accumulated product is coprime to the order of the next factor.
\end{proof}

\begin{remark*}
The coprimality hypothesis is not merely a technicality in the present argument: it is what permits the product labeling to be viewed as a cyclic labeling. Proposition~\ref{prop:weak-product-energy} does not prove closure under arbitrary weak products. Consequently, unrestricted closure of the equality class under the weak product has not been established. Whether such closure holds is a natural open question.
\end{remark*}

\subsection{A related operation: graph joins}

\begin{definition}
If $G$ and $H$ have disjoint vertex sets, their join $G\vee H$ is obtained from their disjoint union by adding every edge between $V(G)$ and $V(H)$.
\end{definition}

\begin{proposition}\label{prop:join-energy}
Let $G$ and $H$ have orders $N_1,N_2$ and edge counts $s_G,s_H$. Then
\[
\mathcal E(G)+\mathcal E(H)-2\sqrt{N_1N_2}
\leq
\mathcal E(G\vee H)
\leq
\mathcal E(G)+\mathcal E(H)+2\sqrt{N_1N_2}.
\]
Consequently,
\[
\operatorname{FR}_{\min}(G\vee H)
\geq
\frac{\max\{0,\mathcal E(G)+\mathcal E(H)-2\sqrt{N_1N_2}\}}
{\sqrt{2(s_G+s_H+N_1N_2)}}.
\]
\end{proposition}
\begin{proof}
With the natural block ordering,
\[
A_{G\vee H}
=
\begin{pmatrix}A_G&0\\0&A_H\end{pmatrix}
+
\begin{pmatrix}0&J\\J^T&0\end{pmatrix},
\]
where $J$ is the $N_1\times N_2$ all-ones matrix. The second summand has two nonzero eigenvalues, $\sqrt{N_1N_2}$ and $-\sqrt{N_1N_2}$, and therefore nuclear norm $2\sqrt{N_1N_2}$. The triangle and reverse-triangle inequalities for the nuclear norm give the energy estimate. The Fourier-ratio estimate follows from Theorem~\ref{thm:energy-bound} and the fact that the join has $s_G+s_H+N_1N_2$ edges.
\end{proof}

\section{Enumeration of low-complexity graphs}\label{section:enumeration}

In this section, we use Theorem~\ref{thm:recovery} in reverse. A sparse sample that recovers an adjacency matrix provides a short code for that matrix: the sampled bits determine a decoder output, and the recovery estimate confines the original matrix to a small Hamming ball around a rounded version of that output. Averaging over all matrices in a low-complexity class produces one sample that works simultaneously for a positive proportion of the class.

A related entropy problem for general signal classes was studied recently in \cite{IosevichHovhannisyanKeyshamsVagharshakyan}. That work obtains upper and lower metric-entropy bounds, as well as uniform sampling results, for layers of functions on $\mathbb Z_N$ with prescribed Fourier ratio. The present argument has a different scope and conclusion. We restrict to the finite and highly constrained class of graph adjacency matrices, incorporate minimization over vertex labelings, and use graph recovery together with Hamming-ball counting to estimate the number of low-edge-complexity graphs. These bounds are compared with the concentration argument for random graphs at the end of Section~\ref{section:random}.

For integers $N\geq2$, $1\leq q\leq\binom{N}{2}$, and $r>0$, let $\mathcal A_{N,q}(r)$ denote the set of adjacency matrices of labeled simple graphs on $\mathbb Z_N$ with exactly $q$ edges and Fourier ratio at most $r$. All logarithms from this point onward are natural.

\begin{lemma}\label{lem:common-sample}
Let $\mathcal C$ be a finite class of functions $f:\mathbb Z_N^2\to\{0,1\}$ with a common value of $\|f\|_2$. Suppose that Theorem~\ref{thm:recovery}, with fixed parameters $p$ and $\epsilon$, succeeds for every $f\in\mathcal C$ with probability at least $3/4$. Put $m=pN^2$. Then there is a deterministic set $X\subset\mathbb Z_N^2$ such that
\[
|X|\leq4m
\]
and the recovery conclusion holds for at least half of the functions in $\mathcal C$.
\end{lemma}
\begin{proof}
If $\mathcal C$ is empty, the conclusion is immediate. Assume henceforth that $\mathcal C$ is nonempty.

For a random sample $X$, let $\alpha(X)$ be the proportion of $f\in\mathcal C$ for which recovery succeeds. Averaging first in $f$ gives $\mathbb E\alpha(X)\geq3/4$. Also $\mathbb E|X|=m$, so Markov's inequality gives $\mathbb P(|X|>4m)\leq1/4$. Therefore
\[
\mathbb E\bigl[\alpha(X)\mathbf 1_{\{|X|\leq4m\}}\bigr]
\geq\frac12.
\]
At least one realization has both asserted properties.
\end{proof}

\begin{lemma}\label{lem:rounding-code}
Fix $q$, $\epsilon$, and a sample $X$. Suppose recovery succeeds on a subcollection $\mathcal C_X\subset\mathcal A_{N,q}(r)$. Then
\[
|\mathcal C_X|
\leq
2^{|X|}
\sum_{j=0}^{R}\binom{N^2}{j},
\qquad
R=\min\left\{N^2,\left\lceil C\epsilon^2N^2\right\rceil\right\},
\]
where $C$ is an absolute constant.
\end{lemma}
\begin{proof}
For every binary data vector $y\in\{0,1\}^X$, choose, by a fixed deterministic rule, one minimizer $D_X(y)$ of
\[
\min_g\|\widehat g\|_1
\quad\text{subject to}\quad
\|g-y\|_{\ell^2(X)}\leq\epsilon\sqrt{2q}.
\]
Thus there are at most $2^{|X|}$ possible decoder outputs. If recovery succeeds for $A\in\mathcal C_X$, then
\[
\|D_X(A|_X)-A\|_2
\leq11.47\epsilon\sqrt{2q}
\leq11.47\epsilon N.
\]
Round every entry of the decoder output to the nearest element of $\{0,1\}$, using a fixed rule at ties. If the rounded matrix differs from $A$ in $h$ positions, each such position contributes at least $1/4$ to the squared Frobenius error. Hence $h\leq C\epsilon^2N^2$. For each decoder output, the number of possible matrices $A$ is bounded by the indicated Hamming-ball volume.
\end{proof}

\subsection{A uniform enumeration bound}

\begin{theorem}\label{thm:enumeration-coarse}
There are absolute constants $c,C>0$ such that, whenever $r\geq\sqrt2$ and
\[
r(\log N)^2\leq cN,
\]
one has, uniformly in $q$,
\[
|\mathcal A_{N,q}(r)|
\leq
2^{CrN(\log N)^2}.
\]
\end{theorem}
\begin{proof}
Set
\[
\epsilon^2=\frac{r\log N}{N},
\qquad
m=KrN(\log N)^2,
\qquad
p=\frac{m}{N^2},
\]
where $K$ is a sufficiently large absolute constant. The hypothesis, after decreasing $c$, ensures $0<\epsilon<1$ and $0<p\leq1$. Moreover,
\[
\frac{r^2}{\epsilon^2}
\left(\log\frac r\epsilon\right)^2\log N
\leq
C\frac{rN}{\log N}(\log N)^3
=CrN(\log N)^2.
\]
Thus Theorem~\ref{thm:recovery} succeeds for every matrix in $\mathcal A_{N,q}(r)$ with probability at least $3/4$ for all sufficiently large $N$.

Lemma~\ref{lem:common-sample} supplies a set $X$ of cardinality at most $4m$ that recovers at least half the class. Lemma~\ref{lem:rounding-code} applies with
\[
R\leq CrN\log N.
\]
Since $R\leq N^2/2$, the standard binomial estimate gives
\[
\sum_{j=0}^{R}\binom{N^2}{j}
\leq
\left(\frac{eN^2}{R}\right)^R
\leq
2^{CrN(\log N)^2}.
\]
There are at most $2^{|X|}\leq2^{CrN(\log N)^2}$ decoder outputs. Multiplying by two to account for the unrecovered half of the class proves the theorem. The finitely many small values of $N$ are absorbed by increasing $C$.
\end{proof}

\subsection{A sharper entropy bound}

The next result uses a larger recovery tolerance. This decreases the number of sampled entries and is more efficient when $r$ is close to the largest range permitted by the argument.

\begin{theorem}\label{thm:enumeration}
There are absolute constants $c,C>0$ with the following property. Suppose $r\geq\sqrt2$ and
\[
r(\log N)^{3/2}\leq cN.
\]
Then, uniformly in $q$,
\[
|\mathcal A_{N,q}(r)|
\leq
2^{C rN(\log N)^{3/2}
\left(1+\log\frac{N}{r(\log N)^{3/2}}\right)}.
\]
\end{theorem}
\begin{proof}
Set
\[
L=(\log N)^{3/2},
\qquad
\epsilon^2=\frac{rL}{N},
\qquad
m=KrNL,
\qquad
p=\frac{m}{N^2}.
\]
After decreasing $c$, the parameters lie in the required ranges. Also
\[
\frac{r^2}{\epsilon^2}
\left(\log\frac r\epsilon\right)^2\log N
\leq
\frac{rN}{L}(\log N)^3
=rNL.
\]
For sufficiently large $K$, Theorem~\ref{thm:recovery} therefore succeeds with probability at least $3/4$ for every matrix in the class.

By Lemma~\ref{lem:common-sample}, a deterministic set $X$ with $|X|\leq4m$ recovers at least half the class. Lemma~\ref{lem:rounding-code} gives a Hamming radius
\[
R\leq CrNL.
\]
The assumption $rL\leq cN$ ensures $R\leq N^2/2$, and hence
\[
\sum_{j=0}^{R}\binom{N^2}{j}
\leq
\left(\frac{eN^2}{R}\right)^R
\leq
2^{CrNL\left(1+\log\frac{N}{rL}\right)}.
\]
The factor $2^{|X|}$ and the factor two for the unrecovered half are absorbed into the same bound.
\end{proof}

Let $\mathcal A_N(r)=\bigcup_q\mathcal A_{N,q}(r)$. Let $\mathcal G_{N,q}(r)$ be the set of labeled graphs on the fixed vertex set $[N]$ with $q$ edges and minimum Fourier ratio at most $r$.

\begin{corollary}\label{cor:all-edge-counts}
Under the hypotheses of either Theorem~\ref{thm:enumeration-coarse} or Theorem~\ref{thm:enumeration}, the same bound, after changing the absolute constant, holds for $|\mathcal A_N(r)|$.
\end{corollary}
\begin{proof}
There are at most $\binom{N}{2}+1\leq N^2$ possible edge counts, and this polynomial factor is absorbed into either exponential estimate.
\end{proof}

\begin{corollary}\label{cor:labeled-enumeration}
Under the assumptions of Theorem~\ref{thm:enumeration},
\[
|\mathcal G_{N,q}(r)|
\leq
N!\cdot 2^{C rN(\log N)^{3/2}
\left(1+\log\frac{N}{r(\log N)^{3/2}}\right)}.
\]
The number of isomorphism classes satisfying the same conditions is bounded by the same exponential expression without the factor $N!$.
\end{corollary}
\begin{proof}
Choosing one minimizing labeling from each isomorphism class gives an injection from the relevant isomorphism classes into $\mathcal A_{N,q}(r)$. Hence the number of such isomorphism classes is at most $|\mathcal A_{N,q}(r)|$. Each isomorphism class has at most $N!$ labeled realizations, which gives the stated bound for $|\mathcal G_{N,q}(r)|$.
\end{proof}

\begin{remark*}
The proof gives more than a cardinality estimate. For a fixed edge count, at least half of the low-complexity adjacency matrices lie in Hamming balls of radius $O(rN(\log N)^{3/2})$ around at most $2^{O(rN(\log N)^{3/2})}$ rounded decoder outputs. Using Chernoff concentration for $|X|$ together with the exponentially small recovery-failure probability, the fraction one half can be replaced by $1-o(1)$. Thus small Fourier ratio imposes a strong metric-entropy constraint even without a complete structural classification.
\end{remark*}

\subsection{A lower bound for low-complexity adjacency matrices}

We complement the preceding upper bounds with a general lower bound. The argument uses the stability of the Fourier ratio under a small number of edge edits.

\begin{lemma}\label{lem:stability-enumeration}
Let $G$ and $H$ be labeled graphs of the same order $N$, with adjacency matrices $A_G$ and $A_H$ and positive sizes $s_G$ and $s_H$. If the edge edit distance between $G$ and $H$ is $k$, then
\[
\left|\operatorname{FR}(A_G)-\operatorname{FR}(A_H)\right|
\leq
\frac{2N\sqrt{k}}{\sqrt{\min\{s_G,s_H\}}}.
\]
\end{lemma}
\begin{proof}
Let $\Delta=A_G-A_H$. Since one edge edit changes two symmetric matrix entries, $\Delta$ has exactly $2k$ nonzero entries, each equal to $1$ or $-1$. Hence
\[
\|\Delta\|_2=\sqrt{2k}.
\]
By the reverse triangle inequality, Cauchy--Schwarz, and Parseval's identity,
\begin{align*}
\left|\|\widehat{A_G}\|_1-\|\widehat{A_H}\|_1\right|
&\leq \|\widehat\Delta\|_1\\
&\leq N\|\widehat\Delta\|_2\\
&=N\|\Delta\|_2\\
&=N\sqrt{2k}.
\end{align*}
Assume without loss of generality that $s_H\leq s_G$. Since $|s_G-s_H|\leq k$,
\[
\sqrt{2s_G}-\sqrt{2s_H}\leq\sqrt{2k}.
\]
Also, every function on $\mathbb Z_N^2$ has Fourier ratio at most $N$. Therefore
\begin{align*}
\left|\operatorname{FR}(A_G)-\operatorname{FR}(A_H)\right|\sqrt{2s_H}
&\leq
\left|\operatorname{FR}(A_G)\sqrt{2s_G}
-
\operatorname{FR}(A_H)\sqrt{2s_H}\right|\\
&\quad+
\operatorname{FR}(A_G)
\left(\sqrt{2s_G}-\sqrt{2s_H}\right)\\
&\leq 2N\sqrt{2k}.
\end{align*}
Dividing by $\sqrt{2s_H}$ proves the result.
\end{proof}

\begin{theorem}\label{thm:enumeration-lower}
Let $r_N$ be a sequence of real numbers such that
\[
r_N\to\infty
\qquad\text{and}\qquad
\frac{r_N}{N}\to0.
\]
There is an absolute constant $c>0$ such that, for all sufficiently large $N$,
\[
|\mathcal A_N(r_N)|
\geq
2^{c r_N^2\log(N/r_N)}.
\]
\end{theorem}
\begin{proof}
Let
\[
k_N=\left\lfloor\frac{r_N^2}{25}\right\rfloor,
\]
and let $\mathcal S_N$ be the set of adjacency matrices obtained by deleting exactly $k_N$ edges from the complete graph $K_N$ under its natural labeling. A direct computation gives
\[
\operatorname{FR}(K_N)=2\sqrt{1-\frac{1}{N}}.
\]
Since $r_N=o(N)$, every graph represented by a matrix $A\in\mathcal S_N$ has at least $N^2/4$ edges for all sufficiently large $N$. Lemma~\ref{lem:stability-enumeration} then gives
\begin{align*}
\operatorname{FR}(A)
&\leq
\operatorname{FR}(K_N)
+
\frac{2N\sqrt{k_N}}{\sqrt{N^2/4}}\\
&\leq
2+4\sqrt{k_N}\\
&\leq
2+\frac{4}{5}r_N\\
&<r_N
\end{align*}
for all sufficiently large $N$. Thus $\mathcal S_N\subset\mathcal A_N(r_N)$.

Put $M=\binom{N}{2}$. Since $r_N\to\infty$ and $r_N=o(N)$, we have $1\leq k_N\leq M/2$ for all sufficiently large $N$. For all sufficiently large $N$, we also have
\[
k_N\geq\frac{r_N^2}{50}
\qquad\text{and}\qquad
\frac{M}{k_N}\geq c_0\left(\frac{N}{r_N}\right)^2
\]
for an absolute constant $c_0>0$. The standard estimate
\[
\binom{M}{k_N}\geq\left(\frac{M}{k_N}\right)^{k_N}
\]
therefore gives
\[
|\mathcal A_N(r_N)|
\geq
|\mathcal S_N|
=
\binom{M}{k_N}
\geq
2^{c r_N^2\log(N/r_N)}
\]
for a suitable absolute constant $c>0$.
\end{proof}

\section{Fourier complexity of random graphs}\label{section:random}

We write \(G(N,p)\) for the Erd\H{o}s--R\'enyi random graph on the labeled vertex set \([N]\); see \cite{Bollobas2001} for standard background. Gupta and Iosevich asked whether, for fixed \(p\in(0,1)\), there is a constant \(c_p>0\) such that
\[
\operatorname{FR}_{\min}(G(N,p))\geq c_pN
\]
with probability tending to one. We answer this question in the affirmative and, more generally, obtain the optimal linear order whenever \(Np_N/\log N\to\infty\), provided the edge probabilities remain bounded away from one.

\begin{theorem}\label{thm:random-linear}
Let \((p_N)_{N\in\mathbb N}\) be a sequence of edge probabilities such that
\[
\frac{Np_N}{\log N}\longrightarrow\infty
\qquad\text{and}\qquad
\limsup_{N\to\infty}p_N<1.
\]
Equivalently, suppose there is a constant \(\delta_0>0\) such that
\[
p_N\leq1-\delta_0
\]
for all sufficiently large \(N\). Then there is a constant \(c=c(\delta_0)>0\) such that
\[
cN\leq\operatorname{FR}_{\min}(G(N,p_N))\leq N
\]
with probability \(1-o(1)\) as \(N\to\infty\).
\end{theorem}

\begin{corollary}\label{cor:random-fixed-p}
For every fixed \(p\in(0,1)\), there is a constant \(c_p>0\) such that
\[
\operatorname{FR}_{\min}(G(N,p))\geq c_pN
\]
with probability \(1-o(1)\).
\end{corollary}
\begin{proof}
Apply Theorem~\ref{thm:random-linear} to the constant sequence \(p_N=p\), with \(\delta_0=1-p\).
\end{proof}

For a graph \(G\) with \(s\) edges, define
\[
W(G)=\min_{\sigma\in S_N}\|\widehat{f_\sigma}\|_1.
\]
Then
\[
\operatorname{FR}_{\min}(G)=\frac{W(G)}{\sqrt{2s}}.
\]
We regard the adjacency matrix of a graph on \([N]\) as a function of its free edge variables. For
\[
\xi=(\xi_{ij})_{1\leq i<j\leq N}\in\mathbb R^{\binom N2},
\]
let \(A(\xi)\) be the symmetric matrix with zero diagonal and with
\[
A(\xi)_{ij}=A(\xi)_{ji}=\xi_{ij}
\qquad\text{for }i<j.
\]
Thus \(G(N,p)\) corresponds to independent Bernoulli\((p)\) variables \(\xi_{ij}\). Define
\[
\Phi(\xi)=\|FA(\xi)F\|_1.
\]

We use the following concentration inequality.

\begin{theorem}[{\cite[Theorem 3.24]{Wainwright2019}}]\label{thm:convex-concentration}
Let \(X_1,\dots,X_n\) be independent random variables taking values in \([0,1]\), set \(X=(X_1,\dots,X_n)\), and let \(g:\mathbb R^n\to\mathbb R\) be convex and \(L\)-Lipschitz with respect to the Euclidean norm. Then, for every \(t\geq0\),
\[
\mathbb P\left(|g(X)-\mathbb E[g(X)]|\geq t\right)
\leq2\exp\left(-\frac{t^2}{2L^2}\right).
\]
\end{theorem}

\begin{lemma}\label{lem:phi-concentration}
Let \(\xi_{ij}\), \(1\leq i<j\leq N\), be independent Bernoulli\((p)\) random variables. Then, for every \(t\geq0\),
\begin{equation}\label{eq:phi-concentration}
\mathbb P\left(|\Phi-\mathbb E[\Phi]|\geq t\right)
\leq2\exp\left(-\frac{t^2}{4N^2}\right).
\end{equation}
\end{lemma}
\begin{proof}
The map \(\xi\mapsto A(\xi)\) is linear, as is the map \(M\mapsto FMF\). Since the entrywise \(\ell^1\)-norm is convex, \(\Phi\) is convex.

For \(\xi,\eta\in\mathbb R^{\binom N2}\), the triangle inequality gives
\[
|\Phi(\xi)-\Phi(\eta)|
\leq\|F(A(\xi)-A(\eta))F\|_1.
\]
Cauchy--Schwarz and Parseval's identity yield
\[
\|F(A(\xi)-A(\eta))F\|_1
\leq N\|A(\xi)-A(\eta)\|_2.
\]
Since \(A(\xi)-A(\eta)\) is symmetric with zero diagonal,
\[
\|A(\xi)-A(\eta)\|_2^2
=2\sum_{i<j}(\xi_{ij}-\eta_{ij})^2
=2\|\xi-\eta\|_2^2.
\]
Thus \(\Phi\) is \(\sqrt2N\)-Lipschitz. Theorem~\ref{thm:convex-concentration} with \(L=\sqrt2N\) gives \eqref{eq:phi-concentration}.
\end{proof}

To apply Lemma~\ref{lem:phi-concentration}, we need a lower bound for \(\mathbb E[\Phi]\). We begin with an elementary moment inequality.

\begin{lemma}\label{lem:moment-ratio}
Let \(Z\) be a complex random variable such that
\[
0<\mathbb E[|Z|^4]<\infty.
\]
Then
\[
\mathbb E[|Z|]
\geq
\frac{\left(\mathbb E[|Z|^2]\right)^{3/2}}
{\left(\mathbb E[|Z|^4]\right)^{1/2}}.
\]
\end{lemma}
\begin{proof}
Cauchy--Schwarz gives
\[
\mathbb E[|Z|^2]
=\mathbb E[|Z|^{1/2}|Z|^{3/2}]
\leq
\left(\mathbb E[|Z|]\right)^{1/2}
\left(\mathbb E[|Z|^3]\right)^{1/2}
\]
and
\[
\mathbb E[|Z|^3]
=\mathbb E[|Z||Z|^2]
\leq
\left(\mathbb E[|Z|^2]\right)^{1/2}
\left(\mathbb E[|Z|^4]\right)^{1/2}.
\]
Combining these inequalities gives the result.
\end{proof}

We next estimate the moments of an individual Fourier coefficient of the random adjacency matrix.

\begin{lemma}\label{lem:moments}
Let \(A\) be the adjacency matrix of \(G(N,p)\), let
\[
\widehat f=FAF,
\]
and write \(v=p(1-p)\). Suppose \(N\geq4\) and \(vN^2\geq8\). If \((m,n)\in\mathbb Z_N^2\) satisfies
\[
m\not\equiv n\pmod N
\qquad\text{and}\qquad
m+n\not\equiv0\pmod N,
\]
then
\[
\mathbb E[\widehat f(m,n)]=0,
\qquad
\mathbb E[|\widehat f(m,n)|^2]\geq\frac v2,
\qquad
\mathbb E[|\widehat f(m,n)|^4]\leq4v^2.
\]
\end{lemma}
\begin{proof}
Since \(\mathbb E[A_{xy}]=p\) for \(x\neq y\) and \(A_{xx}=0\),
\[
\mathbb E[\widehat f(m,n)]
=\frac pN\left(
\sum_{x,y\in\mathbb Z_N}e^{-2\pi i(mx+ny)/N}
-
\sum_{x\in\mathbb Z_N}e^{-2\pi i(m+n)x/N}
\right).
\]
Both sums vanish under the hypothesis \(m+n\not\equiv0\pmod N\), so
\[
\mathbb E[\widehat f(m,n)]=0.
\]

For \(x<y\), set
\[
X_{xy}=A_{xy}-p
\]
and
\[
c_{xy}=e^{-2\pi i(mx+ny)/N}+e^{-2\pi i(my+nx)/N}.
\]
The variables \(X_{xy}\) are independent and centered, with
\[
\mathbb E[X_{xy}^2]=v
\qquad\text{and}\qquad
|X_{xy}|\leq1,
\]
and \(|c_{xy}|\leq2\). Grouping the two ordered entries associated with each edge gives
\begin{equation}\label{eq:centered-coefficient}
\widehat f(m,n)=\frac1N\sum_{x<y}X_{xy}c_{xy}.
\end{equation}
Independence and centering therefore imply
\[
\mathbb E[|\widehat f(m,n)|^2]
=\frac{v}{N^2}\sum_{x<y}|c_{xy}|^2.
\]
Now
\[
|c_{xy}|^2
=2+2\cos\left(\frac{2\pi}{N}(m-n)(x-y)\right).
\]
Since \(m\not\equiv n\pmod N\), character orthogonality gives
\[
\sum_{x,y\in\mathbb Z_N}
\cos\left(\frac{2\pi}{N}(m-n)(x-y)\right)=0.
\]
After removing the \(N\) diagonal terms and using symmetry in \(x\) and \(y\), we obtain
\[
\sum_{x<y}|c_{xy}|^2=N(N-2).
\]
Consequently,
\[
\mathbb E[|\widehat f(m,n)|^2]
=\frac{v(N-2)}{N}
\geq\frac v2.
\]

For the fourth moment, let \(\mathcal P\) be the set of unordered pairs \(\{x,y\}\) with \(x<y\). For \(e=\{x,y\}\in\mathcal P\), write \(X_e=X_{xy}\) and \(c_e=c_{xy}\). Expanding \eqref{eq:centered-coefficient}, every term vanishes unless each index among \(e_1,e_2,e_3,e_4\) occurs at least twice. The nonzero terms therefore have either two distinct indices, each occurring twice, or one index occurring four times. Taking absolute values and using independence gives
\begin{align*}
\mathbb E[|\widehat f(m,n)|^4]
&\leq
\frac{3v^2}{N^4}\sum_{e\neq e'}|c_e|^2|c_{e'}|^2
+
\frac1{N^4}\sum_e\mathbb E[X_e^4]|c_e|^4\\
&\leq
\frac{3v^2}{N^4}\left(\sum_e|c_e|^2\right)^2
+
\frac{16v}{N^4}\binom N2.
\end{align*}
Here the factor \(3\) accounts for the three pairings of four positions, and we used \(\mathbb E[X_e^4]\leq\mathbb E[X_e^2]=v\). Since
\[
\sum_e|c_e|^2=N(N-2),
\]
we conclude that
\[
\mathbb E[|\widehat f(m,n)|^4]
\leq3v^2+\frac{8v}{N^2}
\leq4v^2,
\]
where the last inequality follows from \(vN^2\geq8\).
\end{proof}

\begin{proposition}\label{prop:expected-l1}
Let \((p_N)_{N\in\mathbb N}\) satisfy the hypotheses of Theorem~\ref{thm:random-linear}. Then there is a constant \(c_0=c_0(\delta_0)>0\) such that
\[
\mathbb E[\Phi]\geq c_0\sqrt{p_N}N^2
\]
for all sufficiently large \(N\).
\end{proposition}
\begin{proof}
The hypotheses imply
\[
p_N(1-p_N)N^2
\geq\delta_0p_NN^2
\longrightarrow\infty.
\]
Thus Lemma~\ref{lem:moments} applies with \(p=p_N\) for all sufficiently large \(N\).

For every \((m,n)\) satisfying
\[
m\not\equiv n\pmod N
\qquad\text{and}\qquad
m+n\not\equiv0\pmod N,
\]
Lemmas~\ref{lem:moment-ratio} and \ref{lem:moments} give
\begin{align*}
\mathbb E[|\widehat f(m,n)|]
&\geq
\frac{\left(p_N(1-p_N)/2\right)^{3/2}}
{\left(4p_N^2(1-p_N)^2\right)^{1/2}}\\
&=
\frac{\sqrt{p_N(1-p_N)}}{4\sqrt2}.
\end{align*}
There are at least \(N^2-2N\geq N^2/2\) such frequencies for \(N\geq4\). Hence
\[
\mathbb E[\Phi]
\geq
\frac{\sqrt{p_N(1-p_N)}N^2}{8\sqrt2}
\geq
\frac{\sqrt{\delta_0}}{8\sqrt2}\sqrt{p_N}N^2.
\]
The result follows with
\[
c_0=\frac{\sqrt{\delta_0}}{8\sqrt2}.
\]
\end{proof}

\begin{proof}[Proof of Theorem~\ref{thm:random-linear}]
Since \(\widehat{f_\sigma}\) has \(N^2\) entries, Cauchy--Schwarz gives
\[
\|\widehat{f_\sigma}\|_1\leq N\|\widehat{f_\sigma}\|_2.
\]
Thus
\[
\operatorname{FR}_{\min}(G)\leq N
\]
for every graph with at least one edge.

Let \(c_0=c_0(\delta_0)\) be the constant from Proposition~\ref{prop:expected-l1}. Fix \(\sigma\in S_N\), and let \(\xi^\sigma\) denote the edge variables of the relabeled graph. The permutation \(\sigma\) induces a permutation of the \(\binom N2\) unordered pairs, so \(\xi^\sigma\) has the same distribution as \(\xi\). Therefore Lemma~\ref{lem:phi-concentration} and Proposition~\ref{prop:expected-l1} give
\[
\mathbb P\left(
\Phi(\xi^\sigma)\leq\frac12c_0\sqrt{p_N}N^2
\right)
\leq
2\exp\left(-\frac{c_0^2p_NN^2}{16}\right)
\]
for all sufficiently large \(N\), uniformly in \(\sigma\).

Since
\[
W(G)=\min_{\sigma\in S_N}\Phi(\xi^\sigma),
\]
the union bound gives
\begin{align*}
\mathbb P\left(
W(G)\leq\frac12c_0\sqrt{p_N}N^2
\right)
&\leq
2N!\exp\left(-\frac{c_0^2p_NN^2}{16}\right)\\
&\leq
2\exp\left(
N\log N-\frac{c_0^2p_NN^2}{16}
\right).
\end{align*}
The last expression tends to zero because
\[
\frac{Np_N}{\log N}\longrightarrow\infty.
\]
Thus, with probability \(1-o(1)\),
\begin{equation}\label{eq:random-numerator-lower}
W(G)>\frac12c_0\sqrt{p_N}N^2.
\end{equation}

Let \(s\) be the number of edges. Then \(s\) is binomial with mean
\[
\mu=\binom N2p_N.
\]
The multiplicative Chernoff bound gives
\[
\mathbb P\left(s\geq\frac32\mu\right)
\leq\exp\left(-\frac{\mu}{12}\right).
\]
Moreover,
\[
\frac34p_NN^2
=\frac{3N}{2(N-1)}\mu
\geq\frac32\mu.
\]
Therefore
\[
\mathbb P\left(2s\geq\frac32p_NN^2\right)
\leq
\exp\left(-\frac{\mu}{12}\right)
\leq
\exp\left(-\frac{p_NN^2}{48}\right),
\]
where the last inequality uses \(\mu\geq p_NN^2/4\) for \(N\geq2\). The right-hand side is \(o(1)\).

Combining this edge-count estimate with \eqref{eq:random-numerator-lower}, we obtain, with probability \(1-o(1)\),
\[
\operatorname{FR}_{\min}(G(N,p_N))
=\frac{W(G)}{\sqrt{2s}}
>
\frac{\frac12c_0\sqrt{p_N}N^2}
{\sqrt{\frac32p_NN^2}}
=
\frac{c_0}{\sqrt6}N.
\]
This proves the theorem with \(c=c_0/\sqrt6\).
\end{proof}

\begin{remark*}
The entropy estimates of Section~\ref{section:enumeration} provide a second, weaker route to random-graph lower bounds. For fixed \(p\in(0,1)\), they imply
\[
\operatorname{FR}_{\min}(G(N,p))
\geq
\frac{N}{(\log N)^{3/2}\log\log N}
\]
with probability \(1-o(1)\). If \(p_N\to0\) and
\[
\frac{Np_N}{(\log N)^{3/2}\log\log N}\longrightarrow\infty,
\]
the same argument gives
\[
\operatorname{FR}_{\min}(G(N,p_N))
\geq
\frac{Np_N}{(\log N)^{3/2}\log\log N}
\]
with probability \(1-o(1)\). Theorem~\ref{thm:random-linear} is stronger both in scale and in its range of sparsity. Thus the polylogarithmic losses belong only to the recovery-based enumeration argument, not to the actual order of the minimum Fourier ratio of a random graph.
\end{remark*}

\section{Relations with classical graph parameters}\label{section:parameters}

We conclude with several consequences of the energy bound.

\begin{proposition}\label{prop:independence}
Let \(G\) be a graph of order \(N\) and positive size \(s\), with minimum degree \(\delta\), spectral radius \(\lambda_1\), and independence number \(\alpha(G)\). Then
\[
\operatorname{FR}_{\min}(G)
\geq
\frac1{\sqrt{2s}}
\left(
\lambda_1+
\frac{\alpha(G)\delta^2}
{\lambda_1(N-\alpha(G))}
\right).
\]
If \(G\) is \(d\)-regular, then
\[
\operatorname{FR}_{\min}(G)
\geq \frac{\sqrt{dN}}{N-\alpha(G)}.
\]
\end{proposition}
\begin{proof}
Let \(\lambda_N\) be the least adjacency eigenvalue. The generalized Hoffman bound \cite[Proposition 3.5.3]{BrouwerHaemers} gives
\[
\alpha(G)\leq
N\frac{-\lambda_1\lambda_N}{\delta^2-\lambda_1\lambda_N}.
\]
Rearranging yields
\[
|\lambda_N|
\geq
\frac{\alpha(G)\delta^2}
{\lambda_1(N-\alpha(G))}.
\]
Since \(\mathcal E(G)\geq\lambda_1+|\lambda_N|\), Theorem~\ref{thm:energy-bound} proves the first assertion. In the regular case, \(2s=dN\) and \(\delta=\lambda_1=d\), which gives the second assertion.
\end{proof}

\begin{proposition}\label{prop:chromatic}
Let \(G\) be a graph of positive size \(s\), with chromatic number \(\chi(G)\), spectral radius \(\lambda_1\), and least adjacency eigenvalue \(\lambda_N\). Then
\[
\operatorname{FR}_{\min}(G)
\geq
\frac{\lambda_1}{\sqrt{2s}}
\frac{\chi(G)}{\chi(G)-1}.
\]
In particular,
\[
\operatorname{FR}_{\min}(G)\geq\frac{\chi(G)}{\sqrt{2s}}.
\]
\end{proposition}
\begin{proof}
Hoffman's chromatic bound \cite[Theorem 3.6.2]{BrouwerHaemers} gives
\[
\chi(G)\geq1-\frac{\lambda_1}{\lambda_N},
\]
so \(|\lambda_N|\geq\lambda_1/(\chi(G)-1)\). The first inequality follows from the energy bound. Applying Wilf's theorem \cite{Wilf} to a connected component whose chromatic number is \(\chi(G)\) gives \(\chi(G)\leq\lambda_1+1\). Substitution yields the second inequality.
\end{proof}

\begin{proposition}\label{prop:clique}
Let \(G\) be a graph of positive size \(s\). If \(\omega(G)\) is its clique number, then
\[
\operatorname{FR}_{\min}(G)
\geq\frac{2(\omega(G)-1)}{\sqrt{2s}}.
\]
\end{proposition}
\begin{proof}
Cauchy interlacing \cite[Corollary 2.5.2]{BrouwerHaemers} gives \(\lambda_1\geq\omega(G)-1\). Since the adjacency matrix has trace zero,
\[
\mathcal E(G)=2\sum_{\lambda_j>0}\lambda_j\geq2\lambda_1.
\]
The result follows from Theorem~\ref{thm:energy-bound}.
\end{proof}
For a graph $G=(V,E)$, write $e(S,V\setminus S)$ for the number of edges with one endpoint in $S$ and the other in $V\setminus S$. Its isoperimetric (Cheeger) constant is
\[
h(G)=\min_{\substack{\varnothing\neq S\subseteq V\\ |S|\leq |V|/2}}
\frac{e(S,V\setminus S)}{|S|}.
\]

\begin{proposition}\label{prop:cheeger}
Let $G\neq K_N$ be a $d$-regular graph of order $N$, where $d>0$. Then
\[
\operatorname{FR}_{\min}(G)\geq \frac{4(d-h(G))}{\sqrt{dN}}.
\]
\end{proposition}
\begin{proof}
First, $\lambda_2(G)\geq0$. If $G$ is disconnected, then $d$ is an adjacency eigenvalue of multiplicity at least two, so $\lambda_2(G)=d$. If $G$ is connected, then, because $G\neq K_N$, it contains an induced copy of $P_3$. Cauchy interlacing gives
\[
\lambda_2(G)\geq\lambda_2(P_3)=0.
\]

The Cheeger inequality for regular graphs \cite[Proposition 4.5.2]{BrouwerHaemers} gives
\[
h(G)\geq\frac{d-\lambda_2(G)}{2},
\]
and hence $\lambda_2(G)\geq d-2h(G)$. Since the adjacency matrix has trace zero and $\lambda_1(G)=d$, we obtain
\begin{align*}
\mathcal E(G)
&=2\sum_{\lambda_j>0}\lambda_j
 \geq 2\bigl(\lambda_1(G)+\lambda_2(G)\bigr)\\
&\geq 2\bigl(d+d-2h(G)\bigr)
 =4\bigl(d-h(G)\bigr).
\end{align*}
The result follows from Theorem~\ref{thm:energy-bound} and $2s=dN$.
\end{proof}

\section{Concluding remarks and open problems}\label{section:conclusion}

The results above show a sharp contrast between structured and generic graphs. Equality in the energy bound forces a weighted partial-permutation structure in Fourier space and, in particular, regularity and even-walk circulancy. At the same time, the affine-involution and Singer constructions show that equality is not confined to cyclic or abelian Cayley graphs. At the opposite extreme, the enumeration bounds show that graphs of small edge complexity occupy a negligible portion of the space of all graphs, while the concentration argument proves that the minimum Fourier ratio of a random graph has the optimal order $N$. The lower enumeration bound shows, however, that the low-complexity class still contains exponentially many perturbations of a single highly structured graph.

Several natural questions remain. First, Theorem~\ref{thm:equality-criterion} is a complete matrix criterion for a fixed labeling, but a purely graph-theoretic characterization of equality is not known. The regularity and component restrictions proved here are necessary but far from sufficient. Second, Proposition~\ref{prop:weak-product-energy} establishes closure under weak products only for coprime graph orders. It is unknown whether the coprimality hypothesis can be removed, or whether the equality-attaining isomorphism classes carry a useful multiplicative structure.

Finally, the projector results suggest a parallel classification problem for harmonic complexity. Even for an edge-extremizing graph, not every spectral projector is forced by the present argument to attain the multiplicity lower bound. Understanding precisely which Fourier blocks contribute to a given eigenspace appears to be the appropriate starting point for that question.

\end{document}